\documentclass{article}
\usepackage[margin=1in]{geometry}
\usepackage[onehalfspacing]{setspace}
\usepackage{amsmath,mathtools}
\usepackage{pgfplots}
\usepackage{subcaption}
\usepackage{booktabs}
\usepackage[braket, qm]{qcircuit}
\usepackage{graphicx}
\usepackage{authblk}
\usepackage{xcolor}
\usepackage{booktabs}
\usepackage{subcaption}
\usepackage[linesnumbered,ruled,vlined]{algorithm2e}
\usepackage{qcircuit}
\usepackage{pgffor}
\usepackage{ifthen}
\usepackage{tcolorbox}

\SetCommentSty{mycommfont}

\SetKwInput{KwInput}{Input}                
\SetKwInput{KwOutput}{Output}              
\SetKwProg{Fn}{Function}{}{}

\usepackage{xinttools}
\usepackage{xintexpr}
\usepackage{multicol}
\usepackage{centeredline}
\usepackage{nicematrix}
\usepackage{arydshln}

\usepackage[utf8]{inputenc}
\usepackage[margin=1in]{geometry}

\usepackage[utf8]{inputenc}
\usepackage{pgfmath}
\usepackage{amsmath}
\usepackage{xparse}

\usepackage{esint}
\usepackage{alltt}
\usepackage{newverbs}
\usepackage{comment}

\usepackage{setspace}
\usepackage{listings}

\usepackage{xcolor}
\usepackage{bm}

\usepackage{braket}

\newcommand{\lob}[1]{\left( #1 \right)} 
\newcommand{\labs}[1]{\left| #1 \right|} 

\newcommand{\iseq}[1]{0,\, 1,\, \ldots\,,\, #1-1}

\newcommand{\xseq}[2]{{#1}_0,\, {#1}_1,\, \ldots \,,\, {#1}_{#2-1}}

\newcommand{\iseqe}[1]{0,\, 1,\, \ldots \,,\, #1}

\newcommand{\nn}{\nonumber}

\ExplSyntaxOn
\NewDocumentCommand{\mymat}{mmmm}
{
	\tl_clear:N \l_tmpa_tl
	\int_step_inline:nnn  {0}{ #2 -1 }
	{
		\int_step_inline:nnn {0}{ #3 -1}
		{
			\tl_put_right:Nn \l_tmpa_tl {$#1{\csname #4 \endcsname {##1}{####1} {#2} {#3}}$}
			\int_compare:nNnTF {####1} < {#3-1} 
			{\tl_put_right:Nn \l_tmpa_tl {&}} {}
		}
		\tl_put_right:Nn \l_tmpa_tl {\\}
	}
	\tl_use:N \l_tmpa_tl
}

\cs_new:Npn \myAij #1#2#3#4{
	\pgfmathtruncatemacro{\val}{#1 * #2 }
	\val
}

\cs_new:Npn \myHadamard #1#2#3#4{
	\bitsetSetDec{A}{#1}
	\bitsetSetDec{B}{#2}
	\bitsetAnd{A}{B} 
	\pgfmathtruncatemacro{\val}{mod(\bitsetCardinality{A}, 2)}
	\ifnum\val=0
	\pgfmathtruncatemacro{\val} {1}		
	\else 
	\pgfmathtruncatemacro{\val} {-1}	
	\fi
	\val
}

\newcommand{\mybMat}[2]{
	\begin{bNiceMatrix}[r]
		\mymat{}{#1}{#1}{#2}
	\end{bNiceMatrix}
}

\ExplSyntaxOff

\xintdefiifunc fA(x,y) := \xintexpr {(x>y) ? {1}{0}}\relax;

\newcommand{\genMatrix}[3]{
	\xintdefiivar Matrix = ndmap(#3,1..#1; 1..#2); 
	\[
	\def\xintexpralignbegin {\begin{bmatrix*}[r]}%
		\def\xintexpralignend {\end{bmatrix*}}%
	\def\xintexpralignlinesep {\noexpand\\}
	\def\xintexpraligninnersep {&}%
	\let\xintexpralignleftbracket\empty \let\xintexpralignleftsep\empty
	\let\xintexpralignrightbracket\empty \let\xintexpralignrightsep\empty
	\xintthealign	\xintiiexpr Matrix\relax
	\]
}

\usepackage[english]{babel}

\usepackage{pdflscape}
\usepackage{fancyvrb}
\usepackage{calc}
\usepackage{rotating}
\usepackage{pgf,tikz}
\usepackage{mathrsfs}
\usetikzlibrary{arrows}
\usepackage{mathtools}

\DeclarePairedDelimiter\floor{\lfloor}{\rfloor}                                              
\usepackage{listings}
\usepackage[]{qcircuit}
\usepackage{braket}
\usepackage{mathtools}
\usepackage{varwidth}
\usepackage{caption}
\usepackage{subcaption}
\usepackage{graphicx}
\usepackage[export]{adjustbox}

\usepackage{float}

\makeatletter
\def\paragraph{\@startsection{paragraph}{4}%
	\z@\z@{-\fontdimen2\font}%
	{\normalfont\bfseries}}
\makeatother

\usepackage{graphicx}
\usepackage{pgffor}
\usepackage{tikz,tikz-cd}
\usetikzlibrary{backgrounds,fit,decorations.pathreplacing}  
\usepackage{booktabs}
\usepackage{enumerate}

\usepackage{amssymb, amsmath, amsthm}

\usepackage{xypic}
\usepackage{scalerel}
\usepackage{ulem}

\usepackage{graphicx}              
\usepackage{amsmath}               
\usepackage{amsfonts}              
\usepackage{amsthm}                
\usepackage{xkeyval}               
\usepackage{amssymb}
\usepackage{enumerate}             
\usepackage{color}
\usepackage{breqn}                 
\usepackage{bm}                    
\usepackage{cite}

\allowdisplaybreaks[1]
\usepackage{nicefrac}
\usepackage{booktabs}

\usepackage{comment}

\usepackage{kpfonts}
\usepackage{stackengine}
\usepackage{calc}
\newlength\shlength
\newcommand\xshlongvec[2][0]{\setlength\shlength{#1pt}%
	\stackengine{-5.6pt}{$#2$}{\smash{$\kern\shlength%
			\stackengine{7.55pt}{$\mathchar"017E$}%
			{\rule{\widthof{$#2$}}{.57pt}\kern.4pt}{O}{r}{F}{F}{L}\kern-\shlength$}}%
	{O}{c}{F}{T}{S}}

\usepackage{nicematrix}

\usepackage{hyperref}
\hypersetup{
	colorlinks   = true,    
	citecolor    = red      
}

\graphicspath{{./figures/}}

\newcommand{\RN}[1]{%
	\textup{\uppercase\expandafter{\romannumeral#1}}%
}

\allowdisplaybreaks

\newcommand{\meqref}[1]{\text{Eq}.~\eqref{#1}}
\newcommand{\mref}[1]{Sec.~$ \!\ref{#1} $}
\newcommand{\mfig}[1]{Fig.~$ \!\ref{#1} $}

\newtheorem{thm}{Theorem}[section]

\newtheorem{lem}[thm]{Lemma}

\newtheorem{cor}[thm]{Corollary}

\newtheorem{defn}[thm]{Definition} 
\newtheorem{example}[thm]{Example} 
\newtheorem{remark}[thm]{Remark}

\newcommand{\RR}{\mathbb{R}}      
\newcommand{\ZZ}{\mathbb{Z}}      

\newcommand{\mat}[4]{\left[\begin{smallmatrix*}[r]
		#1 & #2 \\
		#3 & #4 \\
	\end{smallmatrix*}\right]}

\def\FF{{\mathbb F}}

\def\NN{{\mathbb N}}

\def\RR{{\mathbb R}}

\def\ZZ{{\mathbb Z}}

\def\mm{{\frak m}}
\def\nn{{\frak n}}

\def\xx{{\frak x}}
\def\yy{{\frak y}}

\DeclareMathOperator {\Seqz} {Seq}

\def\<{\langle}
\def\>{\rangle}

\newcommand{\Seq}[1]{\mathcal{S}_{#1}}

\newcommand{\myseqc}{sequency-complete}
\newcommand{\Myseqc}{Sequency-complete}
\newcommand{\myseqi}{sequency-incomplete}

\newcommand{\myseqo}{sequency-ordered}
\newcommand{\Myseqo}{Sequency-ordered}
\newcommand{\Mset}{M_{n}(\FF)}
\newcommand{\Msetg}[1]{M_{#1}(\FF)}

\numberwithin{equation}{section}

\usetikzlibrary{matrix,fit}

\usepackage{pgfplots}
\pgfplotsset{compat=1.17}

\makeatletter
\def\smallunderbrace#1{\mathop{\vtop{\m@th\ialign{##\crcr
				$\hfil\displaystyle{#1}\hfil$\crcr
				\noalign{\kern3\p@\nointerlineskip}%
				\tiny\upbracefill\crcr\noalign{\kern3\p@}}}}\limits}
\makeatother

\newcommand{\bseqg}[3]{\pgfmathparse{int(#2 + 1)} \text{${#1}_{#3}  \ldots {#1}_{\pgfmathresult} {#1}_{#2}$}}

\begin{document}
	\title{On sequency-complete and sequency-ordered matrices}
	\author[1]{Alok Shukla \thanks{Corresponding author.}}
	\author[2]{Prakash Vedula}
	\affil[1]{School of Arts and Sciences, Ahmedabad University, India}
	\affil[1]{alok.shukla@ahduni.edu.in}
	\affil[2]{School of Aerospace and Mechanical Engineering, University of Oklahoma, USA}
	\affil[2]{pvedula@ou.edu}
	
\date{}

	\maketitle

\begin{abstract}
The concept of sequency holds a fundamental significance in signal analysis using Walsh basis functions. In this study, we closely examine the concept of sequency and explore the properties of sequency-complete and sequency-ordered matrices. We obtain results on  cardinalities of sets containing sequency-complete and sequency-ordered matrices of arbitrary sizes.  We present methods for obtaining interesting classes of sequency-complete and sequency-ordered matrices of arbitrary sizes. We also provide results on the sequencies of columns in tensor products involving two or more matrices.

\end{abstract}

\section{Introduction} \label{sec:intro}
Analogous to the notion of frequency in Fourier analysis, the concept of sequency, as defined below, plays a fundamental role in signal analysis based on Walsh basis functions \cite{walsh1923closed, beauchamp1975walsh}. We note that, 
Walsh basis functions $ W_k (x) $ for $  k =0,~1,~2, ~\ldots~ N-1 $  in sequency order are defined as follows
\begin{align}
	W_0(x) &= 1 \quad \text{for } 0 \leq x \leq 1,  \nonumber\\
	W_{2k} (x) &= W_k(2t) + (-1)^k W_k (2x -1 ), \nonumber \\
	W_{2k+1} (x) &= W_k(2t) - (-1)^k W_k (2x -1 ), \nonumber\\
	W_k(x) &= 0 \quad \text{for } x < 0 \text{ and } x >1, \label{eq_def_Walsh}
\end{align}
where $ N $ is an integer of the form $ N = 2^n$. 
For $ N =8 $, Walsh basis functions in sequency order are shown in \mfig{fig_walsh_sequency}. The
number of sign changes (or zero-crossings) for Walsh basis functions increases as the orders of the functions increase.

\begin{figure}
	\centering
	\begin{subfigure}{0.24\textwidth}
		\centering
		\includegraphics[width=\linewidth]{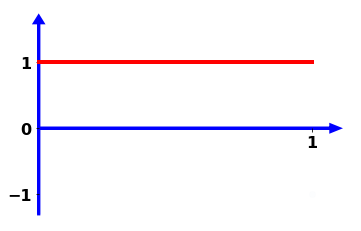}
		\caption{$W_0(x)$}
		\label{fig:WS0}
	\end{subfigure}
	\begin{subfigure}{0.24\textwidth}
		\centering
		\includegraphics[width=\linewidth]{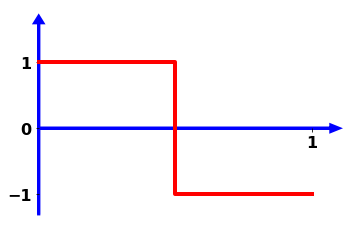}
		\caption{$W_1(x)$}
		\label{fig:WS1}
	\end{subfigure}
	\begin{subfigure}{0.24\textwidth}
		\centering
		\includegraphics[width=\linewidth]{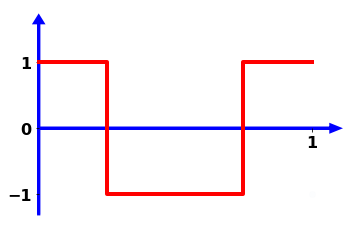}
		\caption{$W_2(x)$}
		\label{fig:WS2}
	\end{subfigure}
	\begin{subfigure}{0.24\textwidth}
		\centering
		\includegraphics[width=\linewidth]{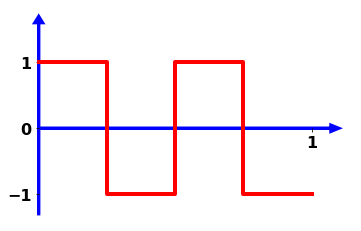}
		\caption{$W_3(x)$}
		\label{fig:WS3}
	\end{subfigure}
	\\
	\begin{subfigure}{0.24\textwidth}
		\centering
		\includegraphics[width=\linewidth]{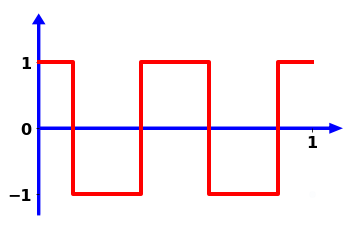}
		\caption{$W_4(x)$}
		\label{fig:WS4}
	\end{subfigure}
	\begin{subfigure}{0.24\textwidth}
		\centering
		\includegraphics[width=\linewidth]{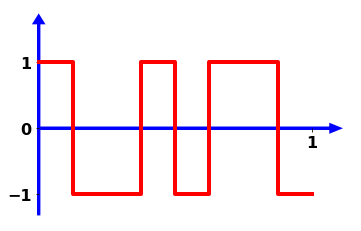}
		\caption{$W_5(x)$}
		\label{fig:WS5}
	\end{subfigure}
	\begin{subfigure}{0.24\textwidth}
		\centering
		\includegraphics[width=\linewidth]{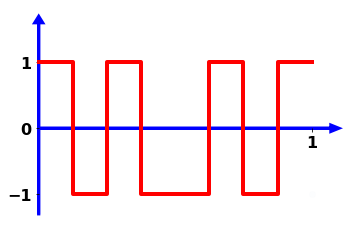}
		\caption{$W_6(x)$}
		\label{fig:WS6}
	\end{subfigure}
	\begin{subfigure}{0.24\textwidth}
		\centering
		\includegraphics[width=\linewidth]{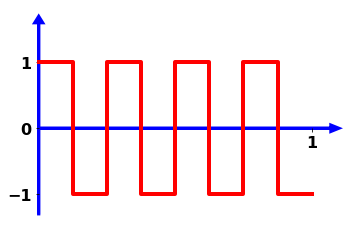}
		\caption{$W_7(x)$}
		\label{fig:WS7}
	\end{subfigure}
 \caption{Walsh basis functions in sequency-order for $N=8$.}
	\label{fig_walsh_sequency}
\end{figure}

 The sequency-ordered (refer Def.~\ref{def:seq_ordered}) Walsh-Hadamard transform matrix  $H_N^S$ is obtained by appropriately arranging the vectors obtained from the sampling of Walsh basis functions as the columns of a matrix. The Walsh-Hadamard transform matrix of order $ N=8 $ in sequency order is
\begin{align} \label{eq:WalshSeq}
H^S_8 = 
	\frac{1}{\sqrt{8}} \,
	\begin{pmatrix*}[r]
		1 & 1 & 1 & 1 & 1 & 1 & 1 & 1  \\
		1 & 1 & 1 & 1 & -1 & -1 & -1 & -1  \\
		1 & 1 & -1 & -1 & -1 & -1 & 1 & 1  \\
		1 & 1 & -1 & -1 & 1 & 1 & -1 & -1  \\
		1 & -1 & -1 & 1 & 1 & -1 & -1 & 1  \\
		1 & -1 & -1 & 1 & -1 & 1 & 1 & -1  \\
		1 & -1 & 1 & -1 & -1 & 1 & -1 & 1  \\
		1 & -1 & 1 & -1 & 1 & -1 & 1 & -1  \\
	\end{pmatrix*}.
\end{align}

More generally, a Hadamard matrix $H$ of order $n$ is defined to be a matrix with orthogonal rows and columns, where each entry is either $1$ or $-1$, and  $ H H^T = n I_n$, where $I_n$ is the identity matrix of order $n$. 
The matrix $H_2 = \mat{1}{1}{1}{-1}$ is an example of Hadamard matrix. The matrix $H_2^{\otimes n}$ is a Hadamard matrix of order $2^n$. However, it is not \myseqo, and it is considered to be in the natural order~\cite{beauchamp1975walsh,geadah1977natural}.

Walsh-Hadamard transform in natural and sequency order have a wide range of applications in diverse fields including quantum algorithms, digital signal and image processing and solution of non-linear differential equations among others \cite{deutsch1992rapid, bernstein1993quantum, shukla2023generalization, simon1997power, grover1996fast, shor1999polynomial,
shuklavedula2019, shukla2023quantum, shukla2024efficient, beer1981walsh}.

In this work, we explore the notion of sequency in more detail and investigate some properties of \myseqc \, and \myseqo \, matrices. We believe that the notion of \myseqc \, matrices is new and has not been studied before. We give the result on cardinalities of the sets of all  \myseqc \, and \myseqo \, matrices of arbitrary size in \mref{sec:sequency}. The notion of sequency defines a natural preorder relation on the column vectors of matrices of arbitrary size. We give a result on the number of maximal ascending chains using this preorder relation in \mref{sec:sequency}.
Several interesting methods of obtaining \myseqc \, and \myseqo \, matrices of arbitrary sizes are also discussed in \mref{sec:sequency}. 
In \mref{sec:tensor},  results on the sequencies of columns of tensor products of two or more matrices are provided.

\subsection{Notation}

\begin{itemize}
    \item $M_{n}(\RR)$ will denote the set of all $n \times n$ matrices over $\RR$.
    \item Unless otherwise specified, $\FF = \{1,-1\}$. $M_{n}(\FF)$ will denote the set of all $n \times n$ matrices whose elements are restricted to lie in $\FF \subseteq \RR$. 
    The subsets of $M_{n}(\FF)$ denoted by $M_{n}(\FF,C)$ and $M_{n}(\FF,O)$, will represent the sets of $n \times n$ \myseqc \, and \myseqo \, matrices, respectively (refer Def.~\ref{def:seq_complete} and \ref{def:seq_ordered}).
    \item For any matrix $A \in M_{n}(\RR)$, its elements are denoted by $A_{i,j}$, where $i, j \in \{\iseq{n}\}$. The $k$-th column of a matrix $A$ will be denoted as $A_{(k)}$. We will let $\Seq{k}$ denote the sequency of the $j$-th column vector of $A$.
    Note that indexing of rows and columns of a matrix starts from $0$. 
    \item Let $v \in \FF^n$. For the column vector $v = [v_0 \, v_1 \, \ldots \, v_{n-1}]^T$, define $p(v) = \frac{1}{2} \labs{v_0 - v_{n-1}}$. The symbol $p_k$ stands for $p(A_{(k)})$, when $A$ is clear from the context.
    \item
    Let $\xseq{k}{n} \in \ZZ/q \ZZ$. 
    Then with $k_{n-1} \otimes \cdots \otimes k_1 \otimes k_0 \in  (\ZZ/q \ZZ)^{\otimes n} $ one can associate a base $q$ number $k_{n-1} \ldots k_1k_0 = q^{n-1} k_{n-1} + \cdots + q k_1 + k_0$. We will interchangeably use the notations
    $k_{n-1} \otimes \cdots \otimes k_1 \otimes k_0 $ and $k_{n-1} \ldots k_1k_0 = q^{n-1} k_{n-1}  + \cdots + q k_1 + k_0$.
    \item For any matrix $A \in M_{n}(\FF)$, the columns of $A^{\otimes n}$ can be indexed by $k_{n-1} \otimes  \cdots \otimes k_1 \otimes k_0 \in  (\ZZ/q \ZZ)^{\otimes n} $, where $k_j \in \ZZ/q \ZZ$ for $j \in \{ \iseq{n}\} $.     
    We will denote by $\Seq{\bseqg{k}{0}{n-1}}$ the sequency of  the column $A_{\bseqg{k}{0}{n-1}}$ of $A^{\otimes n}$. Here we note that the columns of $A^{\otimes n}$ are indexed by $k_{n-1} \otimes  \cdots \otimes k_1 \otimes k_0$ and $A_{\lob{\bseqg{k}{0}{n-1}}}$ denotes the ${k_{n-1} \ldots k_1k_0}$-th or equivalently $(q^{n-1} k_{n-1} + \cdots + q k_1 + k_0)$-th column of $A^{\otimes n}$.
\end{itemize}

\subsection{Preliminaries}

\begin{defn}[\textbf{Sequency}]
   Let $ v \in \FF^n$. The sequency of the column vector  $v = [v_0 \,\,\, v_1 \,\,\, \ldots \,\,\, v_{n-1}]^T$ (or the sequence $\{ v_{k}\}_{k=0}^{n-1}$) is defined as
    \begin{align} \label{def:seq}
        \Seq{v} =  \frac{1}{2} \sum_{k=0}^{n-2} \, \labs{v_{k} - v_{k+1}}.
    \end{align}
\end{defn}
As noted earlier,  sequency gives the count of the number of sign changes in the given column vector (or sequence). 
Additionally, the concept of sequency can be related to the Hamming weight and Hamming distance. We recall that the Hamming weight of a sequence is the number of non-zero elements in the sequence and the Hamming distance between two equal-length sequences of symbols is the number of positions at which the corresponding symbols are different. 
Given the sequence $\{ f_{k}\}_{k=0}^{n-1}$, if we obtain its derived sequence as $\{ df_{k}\}_{k=0}^{n-2}$, where $df_{k} =  \frac{1}{2} \labs{f_{k} - f_{k+1}}$, then  the sequency of $\{ f_{k}\}_{k=0}^{n-1}$ is the Hamming weight of the derived sequence $\{ df_{k}\}_{k=0}^{n-2}$. 
Furthermore, the sequency of $\{ f_{k}\}_{k=0}^{n-1}$ is also the Hamming distance between the sequences $ \{ f_{k}\}_{k=0}^{n-2}$ and $\{ f_{k}\}_{k=1}^{n-1}$.

\ExplSyntaxOn

\newcommand{\myfzero}[4]{{\xintiiexpr (-1)^mod(#1*#2,\xinteval{#3}) \relax}}

\newcommand{\mybNMat}[3]{
\[
{
\NiceMatrixOptions{last-row,nullify-dots, code-for-last-row={\color{blue}}}
\begin{bNiceMatrix}[r]
 \mymat{}{#1}{#1}{#2} 
 #3 \\
\end{bNiceMatrix}
}
\]
}
\ExplSyntaxOff

\begin{defn}[\textbf{\Myseqc}] \label{def:seq_complete}
A matrix $A \in \Mset$ is said to be \textbf{\myseqc} \, if the set
$\{ \Seq{j}   \}_{j=0}^{n-1}$ is in bijection with the set $ \{ j   \}_{j=0}^{n-1}$, where 
$\Seq{j}$ denotes the sequency of the $j$-th column vector of $A$.
A matrix that is not \myseqc \, is said to be \textbf{\myseqi}.
\end{defn}

\begin{remark}
The map 
\begin{equation*}
    \pi : M_n(\RR) \to M_n{(\FF)}
\end{equation*}
such that for $A \in  M_n(\RR)$,  $\pi(A)_{i,j} = \phi(A_{i,j})$ for $i, j \in \iseq{n}$, where 
\begin{equation}
\phi(x) = 
    \begin{cases}
        &  \frac{x}{\labs{x}}  \qquad \text{if $x \neq 0$}, \\ 
        & 1 \qquad \text{otherwise},
    \end{cases}
\end{equation}
can be used to define the notion of sequency for any column vector of a matrix in $M_n(\RR)$. This allows one to extend the concepts of \myseqc, \, \myseqo \, and \myseqi \, matrices corresponding to the set $M_n(\RR)$.
\end{remark}

\begin{example} 
    Let $A \in \Mset$, with $n>1$, such that the $(i,j)$-th element of $A$ is given by
\begin{equation} \label{eq:example1}
A_{i,j}  = (-1)^t, \quad \text{if } \quad { i j \equiv t \mod n}, 
\end{equation} 
where $t \in \{ \iseq{n} \}$.
The matrix $A$ is \myseqi, if $n > 2$ is an even positive integer. If $n$ is an even positive integer, then $\Seq{j} \in \{0,\, n-1\}$. 
If $n$ is an odd positive integer, then $A$ is \myseqc, and one can show that (refer Theorem \ref{thm:Aijzero}):
 \[ \Seq{j} = \begin{cases}
0 & \text{if } j = 0, \\
j-1 & \text{if } j \text{ is even and } 2 \leq j < n, \\
n-j & \text{if } j \text{ is odd and } 1 \leq j <  n.
\end{cases} \]

\end{example}

\begin{figure}
\begin{center}
\begin{tikzpicture}[scale=0.7]
  \matrix (images) [matrix of nodes, nodes in empty cells,
                    row sep= 0.3 em, column sep=2 em,
                    nodes={anchor=center, inner sep=0pt, outer sep=0pt}] {
    \node {$\mybMat{3}{myfzero}$};  & 
    \node {$\mybMat{5}{myfzero}$};  & 
    \node {$ \mybMat{7}{myfzero}$};  \\
      \node{$n=3$,  $ \Seqz_3 = [0, 2, 1]$};  &
    \node {$n=5$, $ \Seqz_5 = [0, 4, 1, 2, 3]$}; &
    \node {$n=7$, $ \Seqz_7 = [0, 6, 1, 4, 3, 2, 5]$}; \\ \\
     \node {$\mybMat{4}{myfzero}$};  &
    \node {$\mybMat{6}{myfzero}$}; &
    \node {$ \mybMat{8}{myfzero}$}; \\
      \node{$n=4$, $ \Seqz_4 = [0, 3, 0, 3]$};  &
    \node {$n=6$, $ \Seqz_6 = [0, 5, 0, 5, 0, 5]$}; &
    \node {$n=8$, $ \Seqz_8 = [0, 7, 0, 7, 0, 7, 0, 7]$}; \\ \\
  };
\end{tikzpicture}

\end{center}
    \caption{The matrices in the top row are \myseqc \, and the matrices in the bottom row are \myseqi. These matrices were generated using \meqref{eq:example1} for $ 3 \leq n \leq 8$. We note that for each $n$, $\Seqz_n$ denotes the sequency of the columns of the corresponding $n \times n$ matrix.
    \label{fig:example1}}
\end{figure}

\begin{defn}[\textbf{\Myseqo}] \label{def:seq_ordered}
A matrix $A \in \Mset$ is said to be \myseqo, if it is \myseqc \, and further
$\Seq{j} = j$, for $j \in \{ \iseq{n} \}$, where 
$\Seq{j}$ denotes the sequency of the $j$-th column vector of $A$.
\end{defn}

Some examples of \myseqo \, matrices are provided in \meqref{eq:WalshSeq}  and \mfig{fig:Seq_ordered}.

\section{Generation of \myseqc \, and \myseqo \, matrices} \label{sec:sequency}
In this section, some interesting classes of \myseqc , and \myseqo , matrices will be discussed. However, before that, the first question we ask is how many \myseqc , and \myseqo , matrices of a given size exist. The answer is provided in the following lemma.
\begin{lem} \label{lem:cardinality}
    There exist $$N_n  = 2^n \prod_{k=1}^{n} k^{(2k - 1 - n)}$$ and $n! N_n$ matrices in $M_n(\FF)$, which are  \myseqo \, and \myseqc, respectively. 
\end{lem}
\begin{proof}
Clearly, $n!$ \myseqc \,  matrices can be obtained by permuting the columns of a \myseqo \, matrix in $M_n(\FF, O)$. Therefore, it is enough to prove the result for \myseqo \, matrices. 
For $A, B \in M_n(\FF, O)$, we define the equivalence relation $\sim$ as follows. We say $ A \sim B$, if the following holds for each $j \in \{\iseq{n}\}$: (a)
$ A_{i,j} = B_{i,j}$ for all $i \in \{\iseq{n}\}$ or (b) $ A_{i,j} = - B_{i,j}$ for all $i \in \{\iseq{n}\}$.  This means  $A \sim B$ if the corresponding columns of $A$ and $B$ are either identical or differ by a multiplicative factor of $-1$.
We will obtain the cardinality of the set $ M_n(\FF, O) / \sim$.
By virtue of the equivalence relation $\sim$, it can be assumed that all the elements of the first rows of any matrix in $M_n(\FF, O) / \sim$ are $1$.
According to \meqref{def:seq}, in any column of size $n$, the number of possible sign changes lies between $0$ and $n-1$, including these values. It is also clear that each term in \meqref{def:seq} contributes either $0$ or $1$ to the total sum. We have also assumed that elements of the first row of the matrix are $1$. 
Therefore, a column of sequency $k$ can be obtained by selecting $k$ terms (out of $n-1$ terms), each contributing $1$. This can be done in $\binom{n-1}{k}$ ways.
Upon multiplying the corresponding numbers $\binom{n-1}{k} $  for each column, we obtain the cardinality of the set  $M_n(\FF, O) / \sim$ to be $\prod_{k=0}^{n-1} \binom{n-1}{k}.$ 
We will show that 
\begin{align} \label{eq:grid}
\prod_{k=0}^{n-1} \binom{n-1}{k} = \prod_{k=1}^{n} k^{(2k - 1 - n)}.
\end{align}
Assuming \meqref{eq:grid}, it is easy to see that the number of \myseqo \, matrices in $M_n(\FF)$ is  $$N_n = 2^n \prod_{k=0}^{n-1} \binom{n-1}{k} .$$  
Now only \meqref{eq:grid} remains to be proved.  
 We note that \meqref{eq:grid} is equivalent to 
 $ \prod_{k=0}^{n-1}  \frac{(n-1)!}{k!}   = \prod_{k=1}^{n-1} k^k$,
which can easily be proved by induction. Alternatively, it can be proved by considering a lower triangular matrix $a$ such that
\begin{align*}
    a_{i,j} = \begin{cases}
         n-1 - j & \qquad \text{if } i \geq j, \\
         0 & \qquad \text{otherwise.}
    \end{cases}
\end{align*}
The equality follows upon computing the product of all the non-zero elements of the matrix $a$ in two different ways and collecting terms together, first along the columns starting from $j=0$ to $j=n-2$ and second along with diagonals and sub-diagonals $i-j = 0$ to $i-j =n-2$. An example for $n-1 = 4$ is given below. 
\[
\begin{bNiceMatrix}[create-medium-nodes, last-row]
4 & 0 & 0 & 0   \\
4 & 3 & 0 & 0    \\
4 & 3&  2& 0    \\
4 & 3 & 2 & 1   \\
\hline
 4^4  & 3^3  & 2^2  & 1^1 
\CodeAfter
\tikz \draw [very thick, red, opacity=0.4] (1-|1.5) -- (last -| 1.5) ;
\tikz \draw [very thick, red, opacity=0.4] (2-|2.5) -- (last -| 2.5) ;
\tikz \draw [very thick, red, opacity=0.4] (3-|3.5) -- (last -| 3.5) ;
\tikz \draw [very thick, red, opacity=0.4] (4-|4.5) -- (last -| 4.5) ;
\end{bNiceMatrix}
 =
\begin{bNiceMatrix}[create-medium-nodes, last-row]
4 & 0 & 0 & 0   \\
4 & 3 & 0 & 0    \\
4 & 3&  2& 0    \\
4 & 3 & 2 & 1   \\
 \hline
 \frac{4!}{3!}  & \frac{4!}{2!}   & \frac{4!}{1!}   & \frac{4!}{0!}  
\CodeAfter
\tikz[very thick, red, opacity=0.4, name suffix = -medium]
\draw (1-1.north west) -- (4-4.south east)
(2-1.north west) -- (4-3.south east)
(3-1.north west) -- (4-2.south east)
(4-1.north west) -- (4-1.south east);
\end{bNiceMatrix}.
\]
\end{proof}

We note that the notion of sequency defines a natural total preorder relation on the columns of matrices in  $\Mset$. These columns belong to the space $\FF^n$.  For $u, v \in \FF^n$, define $ u \preceq v$ if $\Seq{u} \leq \Seq{v} $. Clearly, $\preceq$ is reflexive, and transitive and defines a total preorder in $\FF^n$. Note that the relation $\preceq$ is not antisymmetric in $\FF^n$, so it does not define a partial order. We recall that a \textit{chain} $C$ in a  preorder is its subset such that for any elements $x$ and $y$ in $C$, either $x \preceq y$ or $y \preceq x$. A chain is maximal if it is not a proper subset of any other chain. An ascending chain $(a_0,\, a_1,\, \ldots \,,\, a_{n-1})$ in a preorder is a sequence formed by elements of the chain satisfying the property $ a_i \preceq a_{i+1}$ for $i \in \{ \iseqe{n-2}\}$.

\begin{thm}
    The number of maximal ascending chains in the total preorder $(\FF^n, \preceq )$ is $$\prod_{k=0}^{n-1} \left(2\binom{n-1}{k}\right)! .$$
\end{thm}
\begin{proof}
    It follows from the proof of Lemma \ref{lem:cardinality} that the number of column vectors in  $\FF^n$ of sequency $k$ is given by 
    $2\binom{n-1}{k}$. In any maximal ascending chain, these column vectors can be permuted among themselves in  $(2\binom{n-1}{k})!$ ways. Hence, we can conclude that $(\FF^n, \preceq)$ has a total of $\prod_{k=0}^{n-1} \left(2\binom{n-1}{k}\right)!$ maximal ascending chains.
    \end{proof}
    It is interesting to compare the formula for the number of maximal ascending chains in  $(\FF^n, \preceq )$ given in the above theorem to the results given in \cite{bakoev2020problems} in a different context (see the discussion before Theorem $10$ in \cite{bakoev2020problems}).

\ExplSyntaxOn
\cs_new:Npn \mygetseq #1#2#3{
\pgfmathsetmacro {\xx} {#2} 
  \pgfmathsetmacro \yy {int(\xx - 2)} 
  \foreach \x in {0,...,\yy}
  {%
\xdef \cur {\csname #1 \endcsname {\x}{#3}{#2}{#2}}
\pgfmathsetmacro \z{int(\x + 1)} 
\xdef\mynext{\csname #1 \endcsname {\z}{#3}{#2}{#2}}
\cur
    }
}

\newcommand{\getseq}[3]{%
  \pgfmathsetmacro {\xx} {#2} 
  \pgfmathsetmacro \yy {int(\xx - 2)} 
  \foreach \x in {0,...,\yy}
  {%
   \xdef \cur {\csname #1 \endcsname {\x}{#3}{#2}{#2}}
    \pgfmathsetmacro \z{int(\x + 1)} 
   \xdef\mynext{\csname #1 \endcsname {\z}{#3}{#2}{#2}}
   \mynext
  }%
}

\ExplSyntaxOff

Next, several interesting methods of generating \myseqc \, and \myseqo  \, matrices of arbitrary size are discussed. 

\begin{thm} \label{thm:Aijzero}
     Let $A \in \Mset$, with $n>1$, such that the $(i,j)$-th element of $A$ is given by
\begin{equation} \label{eq:example1}
A_{i,j}  = (-1)^t, \quad \text{if } \quad { i j \equiv t \mod n}, 
\end{equation} 
where $t \in \{ \iseq{n} \}$. If $n$ is an odd positive integer, then $A$ is \myseqc, and in this case 
 \begin{equation} \label{eq:Sijzero}
 \Seq{j} = \begin{cases}
0 & \text{if } j = 0, \\
j-1 & \text{if } j \text{ is even and } 2 \leq j < n, \\
n-j & \text{if } j \text{ is odd and } 1 \leq j <  n.
\end{cases} 
\end{equation}
The matrix $A$ is \myseqi, if $n > 2$ is an even positive integer. If $n$ is an even positive integer then $\Seq{j} \in \{0, n-1\}$, and furthermore in this case
\begin{equation} \label{eq:Sijzeroneven}
\Seq{j} =
    \begin{cases}
       0 & \quad \text{if  $j$ is even}, \\
       n-1 & \quad \text{if  $j$ is odd}.
    \end{cases}
\end{equation}

\end{thm}

\begin{proof}
First, let us consider the case when $n$ is odd. \\
\noindent \textbf{Case 1 ($\bm{n}$ is odd)}: \\ 
In this case, we will obtain  \meqref{eq:Sijzero}, from which it is immediately evident that $A$ is \myseqc.   
We have 
\begin{align*}
    \Seq{j} = \frac{1}{2} \sum_{i=0}^{n-2}  \labs{ (-1)^{r_i} - (-1)^{r_{i+1}}},
\end{align*}
where $ij = q_i n + r_i$, $r_i \in \{\iseq{n}\}$, $q_i \in \ZZ$, and 
$(i+1)j = q_{i+1} n + r_{i+1}$, with $q_{i+1} \in \ZZ$ and
\begin{equation}\label{eq:ti}
    r_{i+1} =
\begin{cases}
     r_i + j & \quad \text{if } r_i + j < n  \quad (\text{or equivalently, $q_{i+1} = q_i$}),\\
      r_i + j - n  & \quad \text{if } r_i + j \geq n \quad  (\text{or equivalently, $q_{i+1} = q_i+1$}).
\end{cases}
\end{equation}
Next, we obtain
\begin{align}
    \Seq{j} & = \frac{1}{2} \sum_{i=0}^{n-2}  \labs{ (-1)^{r_i} - (-1)^{r_{i+1}}}  \nonumber \\ 
    & =  \frac{1}{2} \sum_{r_i + j < n, i \in \{\iseqe{n-2}\}}  \labs{ (-1)^{r_i} - (-1)^{r_{i+1}}} 
    + \frac{1}{2} \sum_{r_i + j \geq n, i \in \{\iseqe{n-2}\}}  \labs{ (-1)^{r_i} - (-1)^{r_{i+1}}}  \nonumber \\
     & =  \frac{1}{2} \sum_{r_i + j < n, i \in \{\iseqe{n-2}\}}  \labs{ (-1)^{r_i} - (-1)^{r_i + j}} 
    + \frac{1}{2} \sum_{r_i + j \geq n, i \in \{\iseqe{n-2}\}}  \labs{ (-1)^{r_i} - (-1)^{r_i + j -n}} \nonumber \\
     & =  \frac{1}{2} \sum_{r_i + j < n, i \in \{\iseqe{n-2}\}}  \labs{ 1 - (-1)^j} 
    + \frac{1}{2} \sum_{r_i + j \geq n, i \in \{\iseqe{n-2}\}}  \labs{ 1 + (-1)^j }.  \label{eq:oddeven}
\end{align}
The last step follows from the observation that $n$ is odd.
Let us first consider the case when $j \geq 2$ is an even number. In this case, only the second term in \meqref{eq:oddeven} survives and we obtain
\begin{align}
    \Seq{j} & =   \frac{1}{2} \sum_{r_i + j \geq n, i \in \{\iseqe{n-2}\}}  \labs{ 1 + 1 }.
\end{align}
It implies that, $\Seq{j} $ is the cardinality of the set $Q$ where 
\begin{align} \label{eq:qi}
 Q =     \{ i \ : \ r_i + j \geq n, i \in \{\iseqe{n-2}\} \} = \{ i \ : \ q_{i+1} = q_i + 1, i \in \{\iseqe{n-2}\} \}. 
\end{align}
We note that as $i$ increases from $i=0$ to $i=n-2$, the sequence $\{q_i\}_{i=0}^{n-2}$ is an increasing sequence with any term either the same as the preceding term or one more than it. Each increase by one in consecutive terms of this sequence corresponds to an element of the set $Q$.  It is obvious from \meqref{eq:ti} and \meqref{eq:qi} that the cardinality of the set $Q$ is given by $q_{n-1}$ (i.e., $q_{i+1}$ when $i=n-2$). We note that $(n-1)j = (j-1) n + n -j$, therefore, $\Seq{j} = q_{n-1} = j-1$.

Next, we consider the case, when $j$ is odd. In this case, only the first term in \meqref{eq:oddeven} contributes and similar to the preceding argument presented for the case when  $j$ was even, one can show that $\Seq{j} = n - j$. Alternatively, it also follows from the observation that the number of the terms in the first sum in \meqref{eq:oddeven} is equal to $n-1 - \#Q = n-1 - (j-1) = n-j$.

Next, let us consider the case when $n$ is even. \\
\noindent \textbf{Case 2 ($\bm{n}$ is even)}: \\ 
We proceed exactly like the preceding case and obtain
\begin{align}
    \Seq{j} & = \frac{1}{2} \sum_{i=0}^{n-2}  \labs{ (-1)^{r_i} - (-1)^{r_{i+1}}}  \nonumber \\ 
      & =  \frac{1}{2} \sum_{r_i + j < n, i \in \{\iseqe{n-2}\}}  \labs{ (-1)^{r_i} - (-1)^{r_i + j}} 
    + \frac{1}{2} \sum_{r_i + j \geq n, i \in \{\iseqe{n-2}\}}  \labs{ (-1)^{r_i} - (-1)^{r_i + j -n}} \nonumber \\
     & =  \frac{1}{2} \sum_{r_i + j < n, i \in \{\iseqe{n-2}\}}  \labs{ 1 - (-1)^j} 
    + \frac{1}{2} \sum_{r_i + j \geq n, i \in \{\iseqe{n-2}\}}  \labs{ 1 - (-1)^j }  \nonumber\\
    & = \frac{1}{2} \sum_{i=0}^{n-2}   \labs{ 1 - (-1)^j }.     \label{eq:oddeventwo}
\end{align}
 \meqref{eq:Sijzeroneven} follows from the last expression and the proof is complete.
\end{proof}
From the preceding discussion,  the following corollary is evident. 
\begin{cor}
   Let $A$ be the matrix as defined in Theorem \ref{thm:Aijzero}. For a fixed integer $j$, with $0 \leq j \leq n -1$, and for a fixed $i$, with $0 < i \leq n -1$,  if $\widetilde{S}_{i,j}$ denotes the number of sign changes in the sequence $ \{A_{s,j}\}_{s=0}^{i} $, then
   \begin{enumerate}
       \item If $n$ is odd, then 
       \[
       \widetilde{S}_{i,j} = 
       \begin{cases}
            q_{i} & \quad \text{if $j$ is even}, \\
             i -  q_{i} &  \quad \text{if $j$ is odd}, \\
       \end{cases}
       \]
    where $q_i = \left\lfloor \frac{ij}{n} \right\rfloor$. An alternate definition of $q_i$ was provided just before \meqref{eq:ti}. 
       \item  If $n$ is even, then 
       \[
       \widetilde{S}_{i,j} = 
       \begin{cases}
           0 &\quad \text{if $j$ is even}, \\
            i & \quad \text{if $j$ is odd}. \\
       \end{cases}
       \]
   \end{enumerate}
  
\end{cor}

\newcommand{\mySeqA}[4]{%
\pgfmathparse {int(#3)}%
\pgfmathtruncatemacro{\nn} {\pgfmathresult}%
\pgfmathtruncatemacro{\mm} {\nn/2}%
\pgfmathtruncatemacro{\val}{mod(\nn, 2)}%
  \ifnum\val=0
		\pgfmathparse{int(mod(#1*#2,\nn) < \mm) ? int(1) : int(-1)}%
  \else 
		\pgfmathparse{int(mod(#1*#2,\nn) < (\mm +1)) ? int(1) : int(-1)}%
    \fi
  \pgfmathresult
}

\begin{figure}
\begin{center}
\begin{tikzpicture}[scale=0.7]
  \matrix (images) [matrix of nodes, nodes in empty cells,
                    row sep= 0.3 em, column sep=2 em,
                    nodes={anchor=center, inner sep=0pt, outer sep=0pt}] {
    \node {$\mybMat{2}{mySeqA}$};  &
    \node {$\mybMat{3}{mySeqA}$};  &
    \node {$ \mybMat{4}{mySeqA}$};  \\
      \node{$n=2$, $ \Seqz_2 = [0, 1]$};  &
    \node {$n=3$, $ \Seqz_3 = [0, 1, 2]$}; &
    \node {$n=4$, $ \Seqz_4 = [0, 1, 3, 2]$}; \\ \\
     \node {$\mybMat{5}{mySeqA}$};  &
    \node {$\mybMat{6}{mySeqA}$}; &
    \node {$ \mybMat{7}{mySeqA}$}; \\
      \node{$n=5$, $ \Seqz_5 = [0, 1, 3, 4, 2]$};  &
    \node {$n=6$, $ \Seqz_6 = [0, 1, 3, 5, 4, 2]$}; &
    \node {$n=7$, $ \Seqz_7 = [0, 1, 3, 5, 6, 4, 2]$}; \\ \\
  };
\end{tikzpicture}

\end{center}
    \caption{Some examples of \myseqc \, matrices. These matrices were generated using \meqref{eq:example2} for $ 2 \leq n \leq 7$.
    We note that for each $n$, $\Seqz_n$ denotes the sequency of the columns of the corresponding $n \times n$ matrix.}
    \label{fig:enter-label}
\end{figure}
 
Next, we provide another interesting class of \myseqc \, matrices.

\begin{thm} \label{thm:Aijone}
    If $A \in \Mset$, with $n>1$, such that the $(i,j)$-th element of $A$ is given by
\begin{equation} \label{eq:example2}
A_{i,j}  = 
\begin{cases}
    1, &  \qquad \text{{if $n=2m$   and}} \quad  ( i j \equiv t \mod n), \\
     1, &  \qquad \text{{if $n=2m+1$   and}} \quad  ( i j \equiv t^\prime \mod n), \\
    -1, & \qquad \text{{otherwise}},
\end{cases}
\end{equation}
where $m$ is a positive integer and $t \in \{ \iseqe{m-1} \}$ and $t^\prime \in \{ \iseqe{m} \}$, then 
the matrix $A$ is \myseqc.
\end{thm}
\begin{proof} 
    Let $\Seq{j}$ denote the sequency of the $j$-th column of $A$, where $j \in \{\iseq{n}\}$. We will show that 
\begin{equation}
\Seq{j}  = 
\begin{cases}
    0, &  \qquad \text{{if $j=0$,}}  \\
     2j -1, &  \qquad \text{if $ 0 < j \leq \left\lfloor \frac{n}{2} \right\rfloor$,} \\
    2(n-j), &  \qquad \text{if $ \left\lfloor \frac{n}{2} \right\rfloor < j \leq n-1$.} \\
\end{cases}
\end{equation}
Since, all the elements of the $0$-th column of $A$ are $1$, clearly $\Seq{0} = 0$.  Next, we consider the case for which $ 0 < j \leq \floor{\frac{n}{2}}$ and show  that $\Seq{j} =  2j -1$. \\
\noindent \textbf{Case 1} ($ 0 < j \leq \floor{\frac{n}{2}}$):
Assume $n =2m$,   $L_0 = \{ \iseq{m} \} $ and $R_0 = \{m,\, m+1,\, \ldots \,,\, n-1 \}$. Let
 $ L = \{ a \in \ZZ \ : \ a \mod n \equiv b, \text{ and }  b \in L_0 \}$ and  
$ R = \{ a \in \ZZ \ : \ a \mod n \equiv b, \text{ and }  b \in R_0 \} $.
Also, let  $ij \equiv r_i \mod n$, such that $r_i \in \{ \iseqe{n-1}\}$.
We have, 
\begin{align*}
    \Seq{j} = \frac{1}{2} \sum_{i=0}^{n-2} \, \labs{A_{i,j} - A_{i+1,j}}.
\end{align*}
It is clear that the only terms contributing to the above sum are those for which  $r_i \in L $ and $ r_i + j \notin L$ or $r_i \notin L $ and $ r_i + j \in L$. In such cases, we say that there is a jump at $i$. Note that, if $0 < r_i \in L $, then $n -r_i \in R$.
Our goal is to count the number of such jumps. We note that for a fixed $j$, $f(i) = ij$ is an increasing function as $i$ goes from $0$ to $n-1$. Another important observation is that as $ 0 < j \leq \floor{\frac{n}{2}}$, if $k \in L_0$, then $ 0 < k + j \leq 2m -1$. It means that if $k$ is in the blue segment, $k+j$ would either remain in the blue segment or, at most, can get to the neighboring red segment on the right (refer to \mfig{fig:LH}). 

Further, we observe that $f(0) = 0$ lies in the first blue segment on the left and the point  $f(n-1) = (n-1)j = (j-1)n + (n-j) $ lies in the last $R$ segment on the right in \mfig{fig:LH}. The image of $f(i)$, as $i$ goes from $0$ to $(n-1)j  $, would contain at least one point from each of the blue and red segments that lie between the end points corresponding to $f(0)$ and $f(n-1)$.
As $(n-1)j= (j-1)n + (n-j)$, it is easy to see that there are  
the $(j-1)$ pairs of blue and red segments (corresponding to the term $(j-1)n$), and an additional pair of blue and red segments corresponding to the term $(n-j)$. Since, we have a total $j$ pairs of blue and red segments, we get that there are $2j-1$ jumps. This shows that $\Seq{j} = 2j -1$, and we are done in this case.

If $n= 2m + 1 $, then the preceding argument still holds if we change the definitions of  $L_0$ and $R_0$ as follows:
$ L_0 = \{ \iseqe{m} \}$ and $R_0 = \{m+1,\, m+2,\, \ldots \,,\, n-1 \}$.

\noindent \textbf{Case 2} ($ \floor{\frac{n}{2}} < j \leq n-1$): The proof in this case is similar to the previous case. We observe that if $0 < r_i \in L $, then $n -r_i \in R$. However, if $r_i=0$, then $r_i \in L$ and $n-r_i \in L $.  
Upon using these observations and interchanging the roles of $L$ and $R$ in Case 1 (if $r_i >0$), it follows  that 
$\Seq{j} = \Seq{n-j} + 1 $, for $ \floor{\frac{n}{2}} < j \leq n-1$. 

\begin{figure}[htbp]
    \centering
    \begin{tikzpicture}[scale=0.9]
        \draw[blue, thick] (0,0) -- node[above] {L} (2,0);
        \draw[red, thick] (2,0) -- node[above] {R} (4,0);
        
        \draw[blue, thick] (4,0) -- node[above] {L} (6,0);
        \draw[red, thick] (6,0) -- node[above] {R} (8,0);
        
        \draw[blue, thick, dotted] (8,0) -- node[above] {L} (10,0);
        \draw[red, thick, dotted] (10,0) -- node[above] {R} (12,0);

        \draw[blue, thick] (12,0) -- node[above] {L} (14,0);
        \draw[red, thick] (14,0) -- node[above] {R} (16,0);

        \foreach \x in {0, 2, 4, 6, 8, 10, 12, 14,16}{
            \filldraw[black] (\x,0) circle (2pt);
            }

        \foreach \x in {2, 3, 4}{
             \node[below] at (2*\x,0) {$\x m$};
            }
         \node[below] at (2,0) {$ m$};

        \node[below] at (12,0) {$2(j-1) m$};  
        \node[below] at (14,0) {$(2j-1)m$};
         \node[below] at (16,0) {$2jm$};

        \node[below] at (0,0) {$0$};
    \end{tikzpicture}
     \caption{Illustration of the segments $L$ (blue) and $R$ (red) as describe in the proof of Theorem \ref{thm:Aijone}. 
     We note that the boundary points $0$, $2m$, $\ldots$, $2jm \in L$ and $m$, $3m$, $\ldots$, $(2j-1)m \in R$.}
    \label{fig:LH}
\end{figure}
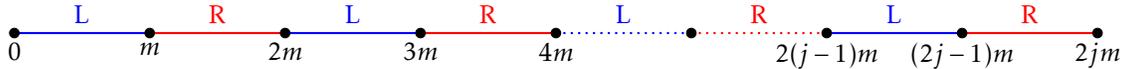
\end{proof}

Finally, the following result captures an interesting method of generating \myseqo \, matrices.
\begin{thm}
    If $A \in \Mset$, with $n>1$, such that the $(i,j)$-th element of $A$ is given by
\begin{equation} \label{eq:example3}
A_{i,j}  = 
\begin{cases}
    1, &  \qquad \text{{if $n=2m$   and  }} ( \floor{i (j+1)/2} \equiv t \mod n, \\
     1, &  \qquad \text{{if $n=2m+1$,   $i(j+1) \equiv 1 \mod 2$}} \text{ and } (  \floor{i(j+1)/2} \equiv t^\prime \mod n), \\
      1, &  \qquad \text{{if $n=2m+1$,   $i(j+1) \equiv 0 \mod 2$}} \text{ and } (  \floor{i(j+1)/2} \equiv t \mod n,\\
 -1, & \qquad \text{{otherwise}},
\end{cases}
\end{equation}
where $m$ is a positive integer and $t \in \{ \iseqe{m-1} \}$ and $t^\prime \in \{ \iseqe{m} \}$, then $A$ is \myseqo \,.
\end{thm}

\newcommand{\mySeqB}[4]{%
\pgfmathtruncatemacro{\nn} {int(#3)}%
\pgfmathtruncatemacro{\result}{mod(#1 * (#2+ 1)/2, \nn) < (\nn/2) ? int(1) : int(-1)}%
  \result
}

\begin{figure}
\begin{center}
\begin{tikzpicture}[scale=1]
  \matrix (images) [matrix of nodes, nodes in empty cells,
                    row sep= 0.3 em, column sep=2 em,
                    nodes={anchor=center, inner sep=0pt, outer sep=0pt}] {
    \node {$\mybMat{2}{mySeqB}$};  &
    \node {$\mybMat{3}{mySeqB}$};  &
    \node {$ \mybMat{4}{mySeqB}$};  \\
      \node{$n=2$};  &
    \node {$n=3$}; &
    \node {$n=4$}; \\ \\
     \node {$\mybMat{5}{mySeqB}$};  &
    \node {$\mybMat{6}{mySeqB}$}; &
    \node {$ \mybMat{7}{mySeqB}$}; \\
      \node{$n=5$};  &
    \node {$n=6$}; &
    \node {$n=7$}; \\ \\
  };
\end{tikzpicture}
\end{center}
    \caption{Some examples of \myseqo \, matrices. These matrices were generated using \meqref{eq:example3} for $ 2 \leq n \leq 7$.}
    \label{fig:Seq_ordered}
\end{figure}
 
\begin{proof} 
    The proof of this theorem is similar to that of Theorem \ref{thm:Aijone}. There are several cases to consider. We give the details only for the case when $n=2m$ and $j+ 1 = 2 q$, with $m, q \in \NN$, $q \geq 1$ and $m \geq 1$.  We note that, in this case, as $i$ runs through $0$ to $n-1$, the function $f(i) = i (j+1)/2 $ (which is an increasing function) takes values from $f(0) = 0$ to $f(n-1) = (n-1)q $. Similar to the proof of Theorem \ref{thm:Aijone}, $\Seq{j}$ can be computed using the values of $f$ at the endpoints, i.e., $f(0) = 0$ and $f(n-1) = (n-1)q$.
   Moreover, $(n-1) q = (q-1)n + n- q$. Note that as $ 1 \leq j \leq n-1$, $ m \leq n- q \leq n-1 $.
   This gives a total of $2 (q-1) + 1 = 2q - 2 + 1 = j$ jumps. Therefore, $\Seq{j} = j$. Other cases are similar and it can be easily verified that $\Seq{j} = j$ holds for each case. 
\end{proof}

\section{Tensor product of \myseqo \, matrices}
\label{sec:tensor}

In this section, we obtain mathematical results related to the sequency of columns of tensor products of two matrices.  It is noteworthy that the tensor product of two \myseqo \, matrices may not necessarily be \myseqo. As an example, consider the matrix $A \otimes A$, 
where \[
A = \mybMat{5}{mySeqB}.
\]
Clearly, $A$ is \myseqo. Further, $A \otimes A$ is a $25 \times 25$ matrix. 
However, one observes that although $A \otimes A$ is \myseqc, it is not \myseqo. 
The sequencies for columns of $A \otimes A$, for the columns $j= \iseqe{24} $ are listed below:
\[
 0, 9, 10, 19, 20, 1, 8, 11, 18, 21, 2, 7, 12, 17, 22, 3, 6, 13, 16, 23, 4, 5, 14, 15, 24. 
\]
If the sequency of the $j$-th column of $A \otimes A$ is denoted by $\Seq{j}$, then 
one can check that the above can be obtained using the following.
\[
\Seq{j} = 
\begin{cases}
  j/5   \quad  \text{if } j  \equiv 0 \mod 5, \\
 9 - (j-1)/5   \quad  \text{if } j  \equiv 1 \mod 5, \\
10 + (j-2)/5   \quad  \text{if } j  \equiv 2 \mod 5, \\
 19 - (j-3)/5   \quad  \text{if } j  \equiv 3 \mod 5, \\  
20 + (j-4)/5   \quad  \text{if } j  \equiv 4 \mod 5.
\end{cases}
\]

The question we ask is: can we give a precise formula for the sequency of any column of $ A^{\otimes n}$ for any integer $n >1$, if the sequency $\Seq{k}$ of the $k$-th column of $A$ is known for $k \in \{\iseq{n}\}$. In the following, $q$ will be assumed to be a positive integer such that $q  \geq 2$.

\begin{lem} \label{lem:seqAA}
    Let $A \in \Msetg{q}$ with $ q \in \ZZ $, $q \geq 2$.  Let 
  $p_k = \frac{1}{2} \labs{A_{0,k} - A_{(q-1),k}}$, where $A_{i,j}$ denotes the $(i,j)$-th element of $A$, with $i$ being the row index and $j$ the column index.
Let $\Seq{k_1k_0}$ denote the sequency of the column $A_{(k_1k_0)}$ of $A \otimes A$.
Then 
\begin{equation}
    \Seq{k_1k_0} = q \Seq{k_0} + (q-1) p_{k_0} + (-1)^{p_{k_0}} \Seq{k_1}.
\end{equation}
\end{lem}
\begin{proof} 
Note that the $(i,k)$-th element of $A \otimes A$ is given by  $ (A \otimes A)_ {i , k } =  A_{i_1,k_1} A_{i_0,k_0} $, where  $i_1i_0 = q i_1 + i_0 $ and $k_1k_0 = q k_1 + k_0  $ are  representations of $i$ and $k$ in base $q$, respectively. 
Therefore, we have
\begin{align} \label{eq:seq_tensor_first}
     \Seq{k_1k_0} & =  \frac{1}2 \sum_{i=0}^{q^2-2}  \labs{(A \otimes A)_ {i , k } - (A \otimes A)_ {i+1 , k }}  \nonumber \\
     & =  \frac{1}2 \sum_{i \nequiv -1 \mod q, \, 0 \leq i \leq q^2 -2}  \labs{(A \otimes A)_ {i , k } - (A \otimes A)_ {i+1 , k }}  + \frac{1}2 \sum_{i \equiv -1 \mod q, \, 0 \leq i \leq q^2 -2}  \labs{(A \otimes A)_ {i , k } - (A \otimes A)_ {i+1 , k }}  \nonumber \\
     & =  \frac{1}2 \sum_{i_1=0}^{q -1} \sum_{i_0 = 0}^{q -2}  \labs{A_{i_1,k_1} A_{i_0,k_0}  - A_{i_1,k_1} A_{i_0+1,k_0} }  + \frac{1}2 \sum_{i_1  = 0}^{q -2}    \labs{A_{i_1,k_1} A_{q-1,k_0}  - A_{i_1+1,k_1} A_{0,k_0}}  \nonumber \\
      & =  \frac{1}2 \sum_{i_1=0}^{q -1} \sum_{i_0 = 0}^{q -2}  \labs{ A_{i_0,k_0}  - A_{i_0+1,k_0} }  + \frac{1}2 \sum_{i_1  = 0}^{q -2}    \labs{A_{i_1,k_1} A_{q-1,k_0}  - A_{i_1+1,k_1} A_{0,k_0}}  \nonumber \\
    & =  q \Seq{k_0}  + \frac{1}2 \sum_{i_1  = 0}^{q -2}    \labs{A_{i_1,k_1} A_{q-1,k_0}  - A_{i_1+1,k_1} A_{0,k_0}}.
\end{align}
    Before proceeding further,  define $L = \{ 0,\, 1,\, \ldots,\, q-2 \} \cap \{j \ : \ A_{j,k_1} = A_{(j+1),k_1} \} $ and $M = \{ 0,\, 1,\, \ldots,\, q-2 \} \cap \{j \ : \ A_{j,k_1} \neq A_{(j+1),k_1} \} $.
    Next, we split the second term on the right hand side of \meqref{eq:seq_tensor_first} into two parts depending upon the condition $A_{j,k_1} = A_{(j+1),k_1}$ or $A_{j,k_1} \neq A_{(j+1),k_1}$.
      We have
    \begin{align} \label{eq:seq_tensor_second}
    & \frac{1}2 \sum_{i_1=0}^{q-2} \labs{A_{i_1,k_1} A_{(q -1),k_0} - A_{(i_1+1),k_1} A_{0,k_0}}  \nonumber \\ & =  \frac{1}2 \sum_{ i_1 \in L} \labs{A_{i_1,k_1} A_{(q -1),k_0} - A_{(i_1+1),k_1} A_{0,k_0}}  + \frac{1}2 \sum_{i_1\in M} \labs{A_{i_1,k_1} A_{(q -1),k_0} - A_{(i_1+1),k_1} A_{0,k_0}}  \nonumber\\
    & = \frac{1}2 \sum_{i_1\in L} \labs{ A_{(q -1),k_0} -  A_{0,k_0}}  + \frac{1}2 \sum_{i_1\in M} \labs{ A_{(q -1),k_0} + A_{0,k_0}} \nonumber \\
    & = p_{k_0} (q- 1 - \Seq{k_1}) + \frac{1}2 \labs{ A_{(q -1),k_0} + A_{0,k_0}} \Seq{k_1} \nonumber \\
    & = (q-1) p_{k_0} + (-1)^{p_{k_0}} \Seq{k_1}.
    \end{align}
    The proof immediately follows from \meqref{eq:seq_tensor_first} and \meqref{eq:seq_tensor_second}. 
\end{proof}

\begin{thm}\label{thm:seqc_A_A}
    If $A \in \Msetg{q}$ is \myseqc, then $A \otimes A$ is \myseqc, i.e., if $A \in M_n(\FF, C) $ then $A \otimes A \in M_n(\FF, C)$. Here $ q \in \ZZ $, $q \geq 2$.
\end{thm}
\begin{proof}
    From Lemma \ref{lem:seqAA}, we have
    \[
      \Seq{k_1k_0} = q \Seq{k_0} + (q-1) p_{k_0} + (-1)^{p_{k_0}} \Seq{k_1}.
    \]
    It is clear that for a fixed $k_0$, the set $\{(q-1) p_{k_0} + (-1)^{p_{k_0}} \Seq{k_1} \ : \ k_1 \in \{0,\, 1, \, \ldots,\, q-1 \} \}$ is in bijection with $\{\iseq{q}\}$. Therefore, as $k_0$ varies from $0$ to $q-1$, the $\Seq{k_1k_0}$ takes distinct values from $0$ to $q^2-1$. It means $A \otimes A \in M_n(\FF, C)$. 
\end{proof}
We note that $\Seq{01} = q \Seq{1} + (q-1) p_{1} (-1)^{p_{1}} \Seq{0}  $.  If $\Seq{0} = 0$ and $\Seq{1} = 1$, then $\Seq{01} = q  + (q-1) p_{1}  \neq 1  $. From this observation, the corollary given below immediately follows.  
\begin{cor}
    If $A \in \Msetg{q}$ is \myseqo, then $A \otimes A$ is not \myseqo.
\end{cor}

The result we proved above for $A \otimes A$ can be generalized and similar results for $A \otimes B$ can be obtained as follows.

\begin{lem} \label{lem:seqAB}
    Let $A \in \Msetg{q_1}$ and $B  \in \Msetg{q_0}$, where $ q_1, q_0 \in \ZZ $, $q_1, q_0 \geq 2$.    Let 
  $p_{m} = \frac{1}{2} \labs{B_{0,m} - B_{q_0-1,m}}$, $m \in \{\iseq{q_0}\}$, where $B_{i,j}$ denotes the $(i,j)$-th element of $B$, with $i$ being the row index and $j$ the column index.
If $\Seq{k}$ denotes the sequency of the column $(A \otimes B)_{(k)}$ of $A \otimes B$, where  $k= k_1 q_0 + k_0$ with $k_1 \in \{\iseq{q_1}\}$ and  $ k_0 \in \{\iseq{q_0}\}$,
then  
\begin{equation} \label{eq_seq_A_cross_B}
    \Seq{k} = q_1 \Seq{k_0} + (q_1-1) p_{k_0} + (-1)^{p_{k_0}} \Seq{k_1}.
\end{equation}
\end{lem}
\begin{proof} 
We note that the $(i,k)$-th element of $A \otimes B$ is given by  $ (A \otimes B)_ {i , k } =  A_{i_1,k_1} B_{i_0,k_0} $, where  $i_1, k_1 \in \{ \iseq{q_1}\} $ and $i_0, k_0 \in \{ \iseq{q_0}\} $, with  $k = q_0 k_1 + k_0  $ and $i = q_0 i_1 + i_0 $. 
Therefore, we have
\begin{align} \label{eq:seq_tensor_first_gen}
     \Seq{k} & =  \frac{1}2 \sum_{i=0}^{q_0 q_1 -2}  \labs{(A \otimes B)_ {i , k } - (A \otimes B)_ {i+1 , k }}  \nonumber \\
     & =  \frac{1}2 \sum_{i \nequiv -1 \mod q_0, \, 0 \leq i \leq q_0q_1 -2}  \labs{(A \otimes B)_ {i , k } - (A \otimes B)_ {i+1 , k }}  + \frac{1}2 \sum_{i \equiv -1 \mod q_0, \, 0 \leq i \leq q_0q_1 -2}  \labs{(A \otimes B)_ {i , k } - (A \otimes B)_ {i+1 , k }}  \nonumber \\
     & =  \frac{1}2 \sum_{i_1=0}^{q_1 -1} \sum_{i_0 = 0}^{q_0 -2}  \labs{A_{i_1,k_1} B_{i_0,k_0}  - A_{i_1,k_1} B_{i_0+1,k_0} }  + \frac{1}2 \sum_{i_1  = 0}^{q_1 -2}    \labs{A_{i_1,k_1} B_{q_0 -1,k_0}  - A_{i_1+1,k_1} B_{0,k_0}}  \nonumber \\
      & =  \frac{1}2 \sum_{i_1=0}^{q_1 -1} \sum_{i_0 = 0}^{q_0 -2}  \labs{ B_{i_0,k_0}  - B_{i_0+1,k_0} }  + \frac{1}2 \sum_{i_1  = 0}^{q_1 -2}    \labs{A_{i_1,k_1} B_{q_0-1,k_0}  - A_{i_1+1,k_1} B_{0,k_0}}  \nonumber \\
    & =  q_1 \Seq{k_0}  + \frac{1}2 \sum_{i_1  = 0}^{q_1 -2}    \labs{A_{i_1,k_1} B_{q_0-1,k_0}  - A_{i_1+1,k_1} B_{0,k_0}}.
\end{align}

    Let $L = \{ 0,\, 1,\, \ldots \, ,\, q_1-2 \} \cap \{j \ : \ A_{j,k_1} = A_{(j+1),k_1} \} $ and $M = \{ 0,\, 1,\, \ldots \,,\, q_1-2 \} \cap \{j \ : \ A_{j,k_1} \neq A_{(j+1),k_1} \} $. We have,
    \begin{align} 
    &  \frac{1}2 \sum_{i_1  = 0}^{q_1 -2}    \labs{A_{i_1,k_1} B_{q_0-1,k_0}  - A_{i_1+1,k_1} B_{0,k_0}}  \nonumber \\ & =  \frac{1}2 \sum_{ i_1 \in L} \labs{A_{i_1,k_1} B_{(q_0 -1),k_0} - A_{(i_1+1),k_1} B_{0,k_0}}  + \frac{1}2 \sum_{i_1\in M} \labs{A_{i_1,k_1} B_{(q_0 -1),k_0} - A_{(i_1+1),k_1} B_{0,k_0}}  \nonumber\\
    & = \frac{1}2 \sum_{i_1\in L} \labs{ B_{(q_0 -1),k_0} -  B_{0,k_0}}  + \frac{1}2 \sum_{i_1\in M} \labs{ B_{(q_0 -1),k_0} + B_{0,k_0}} \nonumber \\
    & = p_{k_0} (q_1- 1 - \Seq{k_1}) + \frac{1}2 \labs{ B_{(q_0 -1),k_0} + B_{0,k_0}} \Seq{k_1} \nonumber \\
    & = (q_1-1) p_{k_0} + (-1)^{p_{k_0}} \Seq{k_1}. \label{eq:seq_tensor_second_gen}
    \end{align}
    The proof immediately follows from \meqref{eq:seq_tensor_first_gen} and \meqref{eq:seq_tensor_second_gen}. 
\end{proof}

Theorem \ref{thm:seqc_A_B} given below follows from Lemma \ref{lem:seqAB}. The proof is similar to that of Theorem \ref{thm:seqc_A_A}.

\begin{thm} \label{thm:seqc_A_B}
    Let $A \in \Msetg{q_1}$ and $B  \in \Msetg{q_0}$, where $ q_1, q_0 \in \ZZ $, $q_1, q_0 \geq 2$. If $A$ and $B$ are \myseqo, then $A \otimes B$ is \myseqc.
\end{thm}

As a consequence of Lemma \ref{lem:seqAB} stated above, we obtain the following result.
\begin{thm}
 Let $A \in \Msetg{q}$, with $ q \in \ZZ $, $q \geq 2$. 
Let $p_k = p(A_{(k)})$ (we recall that for any column vector  $v = [v_0 \,\,\, v_1 \,\,\, \ldots \,\,\, v_{n-1}]^T$ (where $v \in \FF^n$),  $p(v) = \frac{1}{2} \labs{v_0 - v_{n-1}} $).
If $\Seq{\bseqg{k}{0}{n-1}}$ denotes the sequency of the column $A_{\bseqg{k}{0}{n-1}}$ of $A^{\otimes n}$, then we have
\begin{align}
S_{\bseqg{k}{0}{n-1}} &= \sum_{r=0}^{n-1}    (-1)^{ \sum_{m=0}^{r-1} p_{k_m} } \lob{ q^{n-1-r}   \Seq{k_r} + 
(q^{n-1-r} -1) p_{k_r}}.
\end{align}
Here we  treat the expression ${ \sum_{m=0}^{r-1} p_{k_m} }$ as $0$ for $r=0$.
\end{thm}

\begin{proof}
The proof follows from an easy induction argument using Lemma \ref{lem:seqAB}.
\end{proof}

\section{Declarations}
\paragraph{Data availability statement:}
Data sharing is not applicable to this article as no datasets were generated or analyzed during the current study.	

\paragraph{Competing interests statement:}	
The authors have no competing interests to declare that are relevant to the content of this article.

\bibliographystyle{unsrt}

\begin{thebibliography}{10}
	
	\bibitem{walsh1923closed}
	Joseph~L Walsh.
	\newblock A closed set of normal orthogonal functions.
	\newblock {\em American Journal of Mathematics}, 45(1):5--24, 1923.
	
	\bibitem{beauchamp1975walsh}
	Kenneth~George Beauchamp.
	\newblock {\em Walsh functions and their applications}.
	\newblock Academic Press, 1975.
	
	\bibitem{geadah1977natural}
	Youssef~A. Geadah and MJG Corinthios.
	\newblock Natural, dyadic, and sequency order algorithms and processors for the
	{W}alsh--{H}adamard transform.
	\newblock {\em IEEE Transactions on Computers}, 26(05):435--442, 1977.
	
	\bibitem{deutsch1992rapid}
	David Deutsch and Richard Jozsa.
	\newblock Rapid solution of problems by quantum computation.
	\newblock {\em Proceedings of the Royal Society of London. Series A:
		Mathematical and Physical Sciences}, 439(1907):553--558, 1992.
	
	\bibitem{bernstein1993quantum}
	Ethan Bernstein and Umesh Vazirani.
	\newblock Quantum complexity theory.
	\newblock In {\em Proceedings of the twenty-fifth annual ACM symposium on
		Theory of computing}, pages 11--20, 1993.
	
	\bibitem{shukla2023generalization}
	Alok Shukla and Prakash Vedula.
	\newblock A generalization of {B}ernstein--{V}azirani algorithm with multiple
	secret keys and a probabilistic oracle.
	\newblock {\em Quantum Information Processing}, 22(6):244, 2023.
	
	\bibitem{simon1997power}
	Daniel~R Simon.
	\newblock On the power of quantum computation.
	\newblock {\em SIAM journal on computing}, 26(5):1474--1483, 1997.
	
	\bibitem{grover1996fast}
	Lov~K Grover.
	\newblock A {F}ast {Q}uantum {M}echanical {A}lgorithm for {D}atabase {S}earch.
	\newblock In {\em Proceedings of the Twenty-eighth Annual ACM Symposium on
		Theory of Computing}, pages 212--219. ACM, 1996.
	
	\bibitem{shor1999polynomial}
	Peter~W Shor.
	\newblock Polynomial-time algorithms for prime factorization and discrete
	logarithms on a quantum computer.
	\newblock {\em SIAM review}, 41(2):303--332, 1999.
	
	\bibitem{shuklavedula2019}
	Alok Shukla and Prakash Vedula.
	\newblock Trajectory optimization using quantum computing.
	\newblock {\em Journal of Global Optimization}, 75(1):199--225, Sep 2019.
	
	\bibitem{shukla2023quantum}
	Alok Shukla and Prakash Vedula.
	\newblock A quantum approach for digital signal processing.
	\newblock {\em The European Physical Journal Plus}, 138(12):1--24, 2023.
	
	\bibitem{shukla2024efficient}
	Alok Shukla and Prakash Vedula.
	\newblock An efficient quantum algorithm for preparation of uniform quantum
	superposition states.
	\newblock {\em Quantum Information Processing}, 23(2):38, 2024.
	
	\bibitem{beer1981walsh}
	Tom Beer.
	\newblock Walsh transforms.
	\newblock {\em American Journal of Physics}, 49(5):466--472, 1981.
	
	\bibitem{bakoev2020problems}
	Valentin Bakoev.
	\newblock Some problems and algorithms related to the weight order relation on
	the $n$-dimensional boolean cube, 2020.
	
\end{thebibliography}

\end{document}